\documentclass[oneside,english]{amsart}
\usepackage[T1]{fontenc}
\usepackage[latin9]{inputenc}
\usepackage{babel}
\usepackage{enumitem}
\usepackage{amstext}
\usepackage{amsthm}
\usepackage{amssymb}
\usepackage[unicode=true,pdfusetitle,
 bookmarks=true,bookmarksnumbered=false,bookmarksopen=false,
 breaklinks=false,pdfborder={0 0 1},backref=section,colorlinks=false]
 {hyperref}

\makeatletter
\numberwithin{equation}{section}
\numberwithin{figure}{section}
\theoremstyle{plain}
\newtheorem{thm}{\protect\theoremname}[section]
  \theoremstyle{plain}
  \newtheorem{fact}[thm]{\protect\factname}
  \theoremstyle{remark}
  \newtheorem{rem}[thm]{\protect\remarkname}
  \theoremstyle{plain}
  \newtheorem{cor}[thm]{\protect\corollaryname}
  \theoremstyle{plain}
  \newtheorem{prop}[thm]{\protect\propositionname}
  \theoremstyle{remark}
  \newtheorem*{claim*}{\protect\claimname}
  \theoremstyle{plain}
  \newtheorem{lem}[thm]{\protect\lemmaname}
\theoremstyle{lemma}
\newtheorem{mainlemma}[thm]{\protect\mainlemmaname}
 \newlist{casenv}{enumerate}{4}
 \setlist[casenv]{leftmargin=*,align=left,widest={iiii}}
 \setlist[casenv,1]{label={{\itshape\ \casename} \arabic*.},ref=\arabic*}
 \setlist[casenv,2]{label={{\itshape\ \casename} \roman*.},ref=\roman*}
 \setlist[casenv,3]{label={{\itshape\ \casename\ \alph*.}},ref=\alph*}
 \setlist[casenv,4]{label={{\itshape\ \casename} \arabic*.},ref=\arabic*}
  \theoremstyle{definition}
  \newtheorem{problem}[thm]{\protect\problemname}
  \theoremstyle{definition}
  \newtheorem{defn}[thm]{\protect\definitionname}

\usepackage{amsfonts}
\usepackage{nicefrac}
\usepackage{amscd}
\usepackage{a4wide}
\usepackage{url}

\AtBeginDocument{
  
}

\makeatother

  \providecommand{\claimname}{Claim}
  \providecommand{\corollaryname}{Corollary}
  \providecommand{\definitionname}{Definition}
  \providecommand{\factname}{Fact}
  \providecommand{\lemmaname}{Lemma}
  \providecommand{\problemname}{Problem}
  \providecommand{\propositionname}{Proposition}
  \providecommand{\remarkname}{Remark}
 \providecommand{\casename}{Case}
\providecommand{\mainlemmaname}{Main Lemma}
\providecommand{\theoremname}{Theorem}

\begin{document}
\global\long\def\p{\mathbf{p}}
\global\long\def\q{\mathbf{q}}
\global\long\def\C{\mathfrak{C}}
\global\long\def\SS{\mathcal{P}}
 \global\long\def\pr{\operatorname{pr}}
\global\long\def\image{\operatorname{im}}
\global\long\def\otp{\operatorname{otp}}
\global\long\def\dec{\operatorname{dec}}
\global\long\def\suc{\operatorname{suc}}
\global\long\def\pre{\operatorname{pre}}
\global\long\def\qe{\operatorname{qf}}
 \global\long\def\ind{\operatorname{ind}}
\global\long\def\Nind{\operatorname{Nind}}
\global\long\def\lev{\operatorname{lev}}
\global\long\def\Suc{\operatorname{Suc}}
\global\long\def\HNind{\operatorname{HNind}}
\global\long\def\minb{{\lim}}
\global\long\def\concat{\frown}
\global\long\def\cl{\operatorname{cl}}
\global\long\def\tp{\operatorname{tp}}
\global\long\def\id{\operatorname{id}}
\global\long\def\cons{\left(\star\right)}
\global\long\def\qf{\operatorname{qf}}
\global\long\def\ai{\operatorname{ai}}
\global\long\def\dtp{\operatorname{dtp}}
\global\long\def\acl{\operatorname{acl}}
\global\long\def\nb{\operatorname{nb}}
\global\long\def\limb{{\lim}}
\global\long\def\leftexp#1#2{{\vphantom{#2}}^{#1}{#2}}
\global\long\def\intr{\operatorname{interval}}
\global\long\def\atom{\emph{at}}
\global\long\def\I{\mathfrak{I}}
\global\long\def\uf{\operatorname{uf}}
\global\long\def\ded{\operatorname{ded}}
\global\long\def\Ded{\operatorname{Ded}}
\global\long\def\Df{\operatorname{Df}}
\global\long\def\Th{\operatorname{Th}}
\global\long\def\eq{\operatorname{eq}}
\global\long\def\Aut{\operatorname{Aut}}
\global\long\def\ac{ac}
\global\long\def\DfOne{\operatorname{df}_{\operatorname{iso}}}
\global\long\def\modp#1{\pmod#1}
\global\long\def\sequence#1#2{\left\langle #1\,\middle|\,#2\right\rangle }
\global\long\def\set#1#2{\left\{  #1\,\middle|\,#2\right\}  }
\global\long\def\Diag{\operatorname{Diag}}
\global\long\def\Nn{\mathbb{N}}
\global\long\def\mathrela#1{\mathrel{#1}}
\global\long\def\twiddle{\mathord{\sim}}
\global\long\def\mathordi#1{\mathord{#1}}
\global\long\def\Qq{\mathbb{Q}}
\global\long\def\dense{\operatorname{dense}}
\global\long\def\Rr{\mathbb{R}}
 \global\long\def\cof{\operatorname{cof}}
\global\long\def\tr{\operatorname{tr}}
\global\long\def\treeexp#1#2{#1^{\left\langle #2\right\rangle _{\tr}}}
\global\long\def\x{\times}
\global\long\def\forces{\Vdash}
\global\long\def\Vv{\mathbb{V}}
\global\long\def\Uu{\mathbb{U}}
\global\long\def\tauname{\dot{\tau}}
\global\long\def\ScottPsi{\Psi}
\global\long\def\cont{2^{\aleph_{0}}}
\global\long\def\MA#1{{MA}_{#1}}
\global\long\def\rank#1#2{R_{#1}\left(#2\right)}
\global\long\def\cal#1{\mathcal{#1}}

\def\Ind#1#2{#1\setbox0=\hbox{$#1x$}\kern\wd0\hbox to 0pt{\hss$#1\mid$\hss} \lower.9\ht0\hbox to 0pt{\hss$#1\smile$\hss}\kern\wd0} 
\def\Notind#1#2{#1\setbox0=\hbox{$#1x$}\kern\wd0\hbox to 0pt{\mathchardef \nn="3236\hss$#1\nn$\kern1.4\wd0\hss}\hbox to 0pt{\hss$#1\mid$\hss}\lower.9\ht0 \hbox to 0pt{\hss$#1\smile$\hss}\kern\wd0} 
\def\nind{\mathop{\mathpalette\Notind{}}} 

\global\long\def\ind{\mathop{\mathpalette\Ind{}}}
\global\long\def\dom{\operatorname{Dom}}
 \global\long\def\nind{\mathop{\mathpalette\Notind{}}}
\global\long\def\average#1#2#3{Av_{#3}\left(#1/#2\right)}
\global\long\def\Ff{\mathfrak{F}}
\global\long\def\mx#1{Mx_{#1}}
\global\long\def\maps{\mathfrak{L}}

\global\long\def\Esat{E_{\mbox{sat}}}
\global\long\def\Ebnf{E_{\mbox{rep}}}
\global\long\def\Ecom{E_{\mbox{com}}}
\global\long\def\BtypesA{S_{\Bb}^{x}\left(A\right)}

\global\long\def\init{\trianglelefteq}
\global\long\def\fini{\trianglerighteq}
\global\long\def\Bb{\cal B}
\global\long\def\Lim{\operatorname{Lim}}
\global\long\def\Succ{\operatorname{Succ}}

\global\long\def\SquareClass{\cal M}
\global\long\def\leqstar{\leq_{*}}
\global\long\def\average#1#2#3{Av_{#3}\left(#1/#2\right)}
\global\long\def\cut#1{\mathfrak{#1}}
\global\long\def\NTPT{\text{NTP}_{2}}
\global\long\def\Zz{\mathbb{Z}}
\global\long\def\TPT{\text{TP}_{2}}
\global\long\def\supp{\operatorname{supp}}

\global\long\def\OurSequence{\mathcal{I}}

\title{Some remarks on dp-minimal groups}
\begin{abstract}
We prove that $\omega$-categorical dp-minimal groups are nilpotent-by-finite,
a small step in the general direction of proving this for NIP $\omega$-categorical
groups. We also show that in dp-minimal definably amenable groups,
$f$-generic global types are strongly $f$-generic. 
\end{abstract}

\author{Itay Kaplan, Elad Levi,  and Pierre Simon}

\thanks{The first author would like to thank the Israel Science foundation
for partial support of this research (Grant no. 1533/14). }

\thanks{Partially supported by ValCoMo (ANR-13-BS01-0006).}

\address{Elad Levi\\
The Hebrew University of Jerusalem\\
Einstein Institute of Mathematics \\
Edmond J. Safra Campus, Givat Ram\\
Jerusalem 91904, Israel}

\address{Itay Kaplan \\
The Hebrew University of Jerusalem\\
Einstein Institute of Mathematics \\
Edmond J. Safra Campus, Givat Ram\\
Jerusalem 91904, Israel}

\email{kaplan@math.huji.ac.il}

\urladdr{https://sites.google.com/site/itay80/}

\address{Pierre Simon\\
Institut Camille Jordan\\
Université Claude Bernard - Lyon 1\\
43 boulevard du 11 novembre 1918\\
69622 Villeurbanne Cedex, France}

\email{simon@math.univ-lyon1.fr}

\urladdr{http://www.normalesup.org/\textasciitilde{}simon/}

\maketitle

\section{Introduction}

Many results in the model theory of algebraic structures have the
form: if $A$ is an algebraic structure with some model theoretic
property, then $A$ satisfies some nice algebraic properties. There
are many examples of such results, e.g., every $\omega$-stable infinite
field is algebraically closed \cite{Macintyre}. Here our algebraic
structure is a group $G$ and the model theoretic property is dp-minimality,
which we define now. 

A (complete, first order) theory $T$ is \emph{dp-minimal} if the
following cannot happen. There are two formulas $\varphi\left(x,y\right)$,
$\psi\left(x,z\right)$ with $x$ a singleton ($y$ and $z$ perhaps
not), and sequences $\sequence{a_{i}}{i<\omega}$ and $\sequence{b_{j}}{j<\omega}$
such that $\left|a_{i}\right|=\left|y\right|$, $\left|b_{j}\right|=\left|z\right|$
for all $i,j<\omega$ and for every $i,j<\omega$ there is some element
$c_{i,j}$ (all in the monster model $\C\models T$) such that for
all $i',j',i,j<\omega$, $\varphi\left(c_{i,j},a_{i'}\right)$ holds
iff $i=i'$ and $\psi\left(c_{i,j},b_{j'}\right)$ holds iff $j=j'$. 

At first this definition might seem arbitrary, so we will give some
motivation. Recall that $T$ is \emph{NIP} (without the independence
property) or \emph{dependent}, if the following cannot happen. There
is a formula $\varphi\left(x,y\right)$ and sequences $\sequence{a_{i}}{i<\omega}$,
$\sequence{b_{s}}{s\subseteq\omega}$ (all in the monster model $\C$
of $T$) such that $\varphi\left(a_{i},b_{s}\right)$ holds iff $i\in s$. 

NIP plays an important role in current research in model theory. For
more on general NIP, see \cite{pierrebook}. 

Strong dependence is a strengthening of NIP, where one assumes not
only that there is no formula $\varphi\left(x,y\right)$ as in the
definition, but moreover, that there are no $\sequence{\varphi_{i}\left(x,y_{i}\right)}{i<\omega}$
and $\sequence{a_{i,j}}{i,j<\omega}$ in $\C$ (where $\left|a_{i,j}\right|=\left|y_{i}\right|$)
such that for every $\eta:\omega\to\omega$, there is some $b_{\eta}\in\C^{\left|x\right|}$
such that $\varphi_{i}\left(b_{\eta},a_{i,j}\right)$ holds iff $\eta\left(i\right)=j$.
(To see that if $T$ is strongly dependent then it is dependent is
a nice exercise in the definitions.) 

Dp-minimality is then a natural subclass of strong dependence, which
was first properly defined and studied in \cite{OnUs1}. 

Dp-minimal theories are in some sense the simplest case of NIP theories,
but still they include all o-minimal and c-minimal theories and the
theory of the $p$-adics (see \cite[Example 4.28]{pierrebook}). This
restrictive yet still interesting assumption about $T$ yields many
conclusions, evident by the amount of research done in the area, sometimes
with the additional assumption of a group or field structure. See
e.g., \cite{MR3091265,MR2822489,Simon-Dp-min,KaplanSimon,MR3343528,Johnson2015,JahnkeSimonWalsberg2015}
to name a few examples. 

This note contains some results (mostly) on dp-minimal groups, contributing
to the general research in the area. 

In Section \ref{sec:-categorical-dp-minimal-groups} we prove that
all dp-minimal $\omega$-categorical groups are nilpotent-by-finite.
In Subsection \ref{subsec:C(A) abelian for finite A} we prove a general
result on NIP groups: there is a finite $A$ with $C\left(A\right)$
abelian. 

In Section \ref{sec:Boundedly-many-global} we prove that in definably
amenable dp-minimal groups, being $f$-generic is the same as being
strongly $f$-generic.

All definitions are given in the appropriate sections.

\subsubsection*{Acknowledgment. }

We would like to thank the anonymous referee for his review. 

\section{\label{sec:-categorical-dp-minimal-groups}$\omega$-categorical
dp-minimal groups}

\subsection{Introduction}

It is well known that stable $\omega$-categorical groups are nilpotent-by-finite
by \cite{MR533805,MR0491151} where in \cite{MR533805} it is proved
that $\omega$-categorical $\omega$-stable groups are abelian-by-finite
(and it is conjectured that this is true for stable $\omega$-categorical
groups as well). In \cite{MR975912} Macpherson proves that $\omega$-categorical
NSOP groups (and in particular simple in the model theoretic sense)
are also nilpotent-by-finite.

Krupinski generalized the stable case in \cite[Theorem 3.4]{MR2898712}
by proving that that every $\omega$-categorical NIP group that has
fsg (finitely satisfiable generics) is nilpotent-by-finite. In \cite{MR3037553}
Krupinski and Dobrowolski extended this result and removed the NIP
hypothesis\footnote{On the face of it, they asked that the group is generically stable.
However, by \cite[Remark 1.8]{MR2898712}, under NIP and $\omega$-categoricity,
a definable group has fsg iff it is generically stable.}. 

In this section we will go in the other direction and remove the assumption
of fsg. However, our proof requires the stronger assumption of dp-minimality
and not just NIP.

\subsection{What we get from $\omega$-categoricity}

We will need the following facts about $\omega$-categorical theories. 

Suppose that $T$ is $\omega$-categorical. 
\begin{enumerate}
\item (Ryll\textendash Nardzewski, see e.g., \cite[Theorem 4.3.1]{TentZiegler})
For all $n<\omega$, there are at most finitely many $\emptyset$-definable
sets in $n$ variables. 
\item If $M\models T$ is saturated (in particular, countable) and $X\subseteq M^{n}$
is invariant under $\Aut\left(M\right)$ then $X$ is $\emptyset$-definable. 
\item By (2), an $\omega$-categorical theory $T$ eliminates $\exists^{\infty}$,
which means that for all $\varphi\left(x,y\right)$ there is some
$n<\omega$ such that for all $a\in M\models T$, $\varphi\left(M,a\right)$
is infinite iff $\left|\varphi\left(M,a\right)\right|\geq n$. 
\end{enumerate}
A structure $M$ is\emph{ $\omega$-categorical }if its theory is.

Suppose that $\left(G,\cdot\right)$ is an $\omega$-categorical
group. Then, it follows easily from (1) that $\left(G,\cdot\right)$
is locally finite (every finitely generated subgroup is finite). 

We will use the following fact about locally finite groups. 
\begin{fact}
\label{fact:locally finite contains infinite abelian-1}\cite[Corollary 2.5]{MR0470081}
If $G$ is an infinite locally finite group (every finitely generated
subgroup is finite), then $G$ contains an infinite abelian subgroup. 
\end{fact}

\subsection{Equivalent conditions for being nilpotent-by-finite}
\begin{rem}
\label{rem:An elementary property} If $G$ is nilpotent-by-finite,
and $H\equiv G$ then $H$ is nilpotent-by-finite. Why? Suppose that
$\lambda^{+}=2^{\lambda}>\left|G\right|,\left|H\right|$ and let $G^{*}$
be a saturated extension of $G$ of size $\lambda^{+}$, which, we
may assume, also contains $H$. Then it is enough to show that $G^{*}$
is nilpotent-by-finite (being nilpotent-by-finite transfers to subgroups).
Suppose that $G_{0}\leq G$ is nilpotent of finite index. Let $\left(S^{*},S_{0}^{*}\right)$
be a saturated extension of $\left(G,G_{0}\right)$ of size $\lambda$.
Then $S_{0}^{*}\leq S^{*}$ is nilpotent of finite index in $S^{*}$
and $G^{*}\cong S^{*}$.

If there is no such $\lambda$, we can either force its existence
or use special models instead (see \cite[Theorems 10.4.4, 10.4.2]{Hod}).
\end{rem}
Suppose that $G$ is any group. Let $G_{\emptyset}^{00}$ be the intersection
of all $\emptyset$-type-definable subgroups of $G$ (in $\C$). When
$G$ is $\omega$-categorical, it must be $\emptyset$-definable of
finite index (so we can talk about it in $G$ without going to a saturated
extension). However, if $G$ is NIP then, by \cite{Sh876}, $G_{\emptyset}^{00}=G_{A}^{00}$
for any small set $A$.
\begin{fact}
\cite[Theorem 5.2.8]{MR1261639}\label{fact:Fiiting}Let M and N
be normal nilpotent subgroups of a group G, then L = MN is nilpotent
group.
\end{fact}
\begin{cor}
\label{cor:Fitt} Assume $G$ is a nilpotent-by-finite $\omega$-categorical
group with $G=G_{\emptyset}^{00}$, then $G$ is a nilpotent group.

\begin{proof}
By Fact \ref{fact:Fiiting} the product of all normal nilpotent subgroups
of finite index of $G$ is itself a nilpotent group which is also
$\emptyset$-definable (by $\omega$-categoricity) and of finite index,
thus $G$ is nilpotent.
\end{proof}
\end{cor}
\begin{prop}
\label{prop:equivalent conditions for being n-b-f}Suppose that $\cal C$
is a class of countable $\omega$-categorical NIP groups (in the pure
group language) satisfying: if $G\in\cal C$, $H\trianglelefteq G$
definable (over $\emptyset$) then $G/H\in\cal C$ and $H\in\cal C$.
Then the following statements are equivalent:

\begin{enumerate}
\item Every $G\in\cal C$ is nilpotent-by-finite.
\item Every infinite characteristically simple $G\in\cal C$ is abelian.
\item Every infinite $G\in\cal C$ contains an infinite $\emptyset$-definable
abelian subgroup. 
\end{enumerate}
\end{prop}
\begin{proof}
(2) implies (1) is essentially Krupinski's argument from \cite{MR2898712}.
Suppose that $G\in\cal C$ and we wish to show that it is nilpotent-by-finite.
We may of course assume that $G$ is infinite. We may assume that
$G^{00}=G$ so that $G$ has no definable subgroups of finite index.

By $\omega$-categoricity, we can write $\left\{ e\right\} =G_{0}\lneq G_{1}\lneq\cdots\lneq G_{n}=G$
where the groups $G_{i}$ are $0$-definable, and this is a maximal
(length-wise) such chain. The groups $G_{i}$ are invariant under
$\Aut\left(G\right)$ so normal, and by assumption $G_{i}\in\cal C$.
The proof is now by induction on $n\geq1$. For $n=1$, this follows
immediately from (2) (i.e., $G_{1}$ will be abelian). 

Now note that by the induction hypothesis if $H\trianglelefteq G$
is $\emptyset$-definable and non-trivial then $G/H$ is nilpotent:
$G/H$ is in $\cal C$ (as $G=G^{00}$, $G/H$ is infinite). Also
$\left(G/H\right)^{00}=\left(G/H\right)$. But the maximal length
of a chain as above which suits $G/H$ must be shorter than $n$.
Hence $G/H$ is nilpotent-by-finite and by Corollary \ref{cor:Fitt}
$G/H$ is itself nilpotent.  

Hence we may assume that $Z\left(G\right)$ is trivial (otherwise
$G/Z\left(G\right)$ is nilpotent and so $G$ is too). This in turn
implies that $G_{1}$ is infinite (if not, then $C_{G}\left(G_{1}\right)$
is of finite index in $G$, and hence equals $G$, but then $G_{1}\subseteq Z\left(G\right)$).
Now we use (2) on $G_{1}$ to finish: $G_{1}$ is abelian and $G/G_{1}$
is nilpotent, so both are solvable, and hence $G$ is solvable. However,
by \cite[Theorem 1.2]{MR1426520}, if $G$ is not nilpotent-by-finite
(equivalently, nilpotent, since we already assumed $G=G^{00}$), it
interprets the infinite atomless boolean algebra, and has IP. 

(3) implies (2) is obvious. 

(1) implies (3). Without loss of generality, $G=G^{00}$. 

\begin{claim*}
Either $Z\left(G\right)$ is infinite or $G/Z\left(G\right)$ is centerless.
\end{claim*}
\begin{proof}
If $x\in G$ with $y^{-1}xy\in xZ\left(G\right)$ for all $y\in G$,
then $C\left(x\right)$ has a finite index in $G$. Hence $C\left(x\right)=G$
so $x\in Z\left(G\right)$. 
\end{proof}
If $Z\left(G\right)$ is infinite, we are done. Otherwise, by the
claim $G/Z(G)$ is centerless, But by (1) and Corollary \ref{cor:Fitt}
, $G$ is nilpotent so we have a contradiction . 
\end{proof}
\begin{rem}
\label{rem:dp-minimal satisfy the conditions}Note that the class
of all (countable) NIP $\omega$-categorical groups satisfy the conditions
in Proposition \ref{prop:equivalent conditions for being n-b-f}.
So does the class of $\omega$-categorical dp-minimal groups (taking
quotients of the universe $M$, as opposed to e.g., $M^{2}$, preserves
dp-minimality). 
\end{rem}

\begin{rem}
\label{rem:Krupinski-method}In \cite{MR2898712}, Krupinski proved
that (2) in Proposition \ref{prop:equivalent conditions for being n-b-f}
holds for the class of NIP $\omega$-categorical groups with fsg.
His proof uses a classification theorem on $\omega$-stable characteristically
simple groups due to Wilson and Apps \cite{MR679175,MR700288} (see
remarks after Problem \ref{prob:What-about-inp-minimal}). By that
theorem and \cite[Proposition 3.2]{MR2898712}, it follows that $\omega$-categorical
characteristically simple groups with NIP are $p$-groups for some
$p$. By the argument in the proof of \cite[Proposition 3.1]{MR2898712},
it follows that for such groups $G$, if $a_{1},\ldots,a_{n}\in G$
then $C\left(a_{1}\right)\cap\cdots\cap C\left(a_{n}\right)$ is infinite. 

However, we will avoid using the classification theorem, and prove
(3) directly for dp-minimal $\omega$-categorical groups.
\end{rem}

\begin{rem}
Note also that Fact \ref{fact:locally finite contains infinite abelian-1}
alone is not enough, even though it may seem so in light of the fact
that if $G$ is both NIP and contains an infinite abelian subgroup,
then $G$ contains an infinite definable abelian subgroup by \cite[Claim 4.3]{Sh783}.
However this subgroup is not necessarily $\emptyset$-definable.
\end{rem}

\subsection{What we get from dp-minimality and NIP}

The only use of dp-minimality in the proof is the following basic
observation.
\begin{fact}
\label{fact:(Pierre)} \cite[Claim in proof of Proposition 4.31]{pierrebook}
If $\left(G,\cdot\right)$ is a dp-minimal group then for every definable
subgroups $H_{1},H_{2}\leq G$ either $\left[H_{1}:H_{1}\cap H_{2}\right]<\infty$
or $\left[H_{2}:\,H_{1}\cap H_{2}\right]<\infty$.
\end{fact}
We will also use the Baldwin-Saxl lemma, which is true for all NIP
groups. 
\begin{fact}
\label{fact:BaldwinSaxl}\cite{BaSxl} Let $\left(G,\cdot\right)$
be NIP. Suppose that $\varphi\left(x,y\right)$ is a formula and that
$\left\{ \varphi\left(x,c\right)\left|\,c\in C\right.\right\} $ defines
a family of subgroups of $G$. Then there is a number $n<\omega$
(depending only on $\varphi$) such that any finite intersection of
groups from this family is already an intersection of $n$ of them.
\end{fact}

\subsection{Proof of the main result}
\begin{thm}
\label{thm:Main-omega-categorical}If $\left(G,\cdot\right)$ is an
infinite dp-minimal $\omega$-categorical group then $G$ contains
an infinite $\emptyset$-definable abelian subgroup. 
\end{thm}
By Proposition \ref{prop:equivalent conditions for being n-b-f} we
get the following. 
\begin{cor}
If $\left(G,\cdot\right)$ is a dp-minimal $\omega$-categorical group
then $G$ is nilpotent-by-finite. 
\end{cor}
For the proof we work in a countable (so $\omega$-saturated) model.
So fix such a group $G$. By $\omega$-categoricity, there is a minimal
infinite $\emptyset$-definable subgroup $G_{0}\leq G$ (i.e., $G_{0}$
contains no $\emptyset$-definable infinite subgroups), so we may
assume that $G=G_{0}$. 
\begin{lem}
\label{Lem:dp-min applied to centralizers}For every $a,b\in G$ either
$\left[C\left(a\right):C\left(b\right)\cap C\left(a\right)\right]<\infty$
or $\left[C\left(b\right):C\left(b\right)\cap C\left(a\right)\right]<\infty$.

\begin{proof}
This follows directly from Fact \ref{fact:(Pierre)}. 
\end{proof}
\end{lem}
Let $X=\set{a\in G}{\left|C\left(a\right)\right|=\infty}$. By elimination
of $\exists^{\infty}$, $X$ is definable. By Fact \ref{fact:locally finite contains infinite abelian-1},
$X$ is infinite. For $a,b\in X$, by Lemma \ref{Lem:dp-min applied to centralizers},
it follows that either $C\left(a\right)\cap C\left(b\right)$ has
finite index in $C\left(a\right)$ or in $C\left(b\right)$. In either
case, $C\left(a\right)\cap C\left(b\right)$ is infinite. Since $C\left(a\right)\cap C\left(b\right)\subseteq C\left(ab\right)$,
it follows that $X$ is a group. By our assumption on $G$ (it contains
no infinite $\emptyset$-definable subgroups), $G=X$. 

Compare the following corollary with Remark \ref{rem:Krupinski-method}. 
\begin{cor}
\label{cor:every finite intersection is infinite}For every $a_{0},\ldots,a_{n-1}\in G$,
$\bigcap\set{C\left(a_{i}\right)}{i<n}$ is infinite.
\end{cor}
\begin{proof}
By induction on $n$. For $n=1$ and $n=2$ we just gave the argument.
For larger $n$ it is exactly the same: $\bigcap\set{C\left(a_{i}\right)}{i<n}$
has finite index in one of $\bigcap\set{C\left(a_{i}\right)}{i<n-1}$
or $\bigcap\set{C\left(a_{i}\right)}{1\leq i<n}$, both infinite. 
\end{proof}
For every $a\in G$ let $H_{a}=\set{b\in G}{\left[C\left(a\right):C\left(b\right)\cap C\left(a\right)\right]<\infty}$,
and define 
\[
C^{0}\left(a\right)=\bigcap\set{C\left(b\right)}{b\in H_{a}}.
\]
Observe that $C^{0}(a)$ is definable over $a$ since $H_{a}$ is
definable. 

By $\omega$-categoricity there exists some $n_{*}$ such that for
all $a,b\in G$, $\left[C\left(a\right):C\left(b\right)\cap C\left(a\right)\right]<n_{*}$
iff $\left[C\left(a\right):C\left(b\right)\cap C\left(a\right)\right]<\infty$.
\begin{mainlemma}
\label{lem:Main Lemma}For every $a,b\in G$ either $C^{0}\left(a\right)\subseteq C^{0}\left(b\right)$
or $C^{0}\left(b\right)\subseteq C^{0}\left(a\right)$.
\end{mainlemma}
\begin{proof}
By Fact \ref{fact:(Pierre)} either $\left[C\left(a\right):C\left(b\right)\cap C\left(a\right)\right]<n_{*}$
or $\left[C\left(b\right):C\left(b\right)\cap C\left(a\right)\right]<n_{*}$.
Suppose that the former happens. Then $C^{0}\left(a\right)\subseteq C^{0}\left(b\right)$:
if $d\in H_{b}$ , $\left[C\left(b\right):C\left(b\right)\cap C\left(d\right)\right]<n_{*}$,
so 
\begin{eqnarray*}
\left[C\left(a\right):C\left(a\right)\cap C\left(d\right)\right] & \leq & \left[C\left(a\right):C\left(a\right)\cap C\left(d\right)\cap C\left(b\right)\right]\\
 & \leq & \left[C\left(a\right):C\left(a\right)\cap C\left(b\right)\right]\cdot\left[C\left(b\right):C\left(b\right)\cap C\left(d\right)\right]<n_{*}^{2}.
\end{eqnarray*}
Hence $d\in H_{a}$. 
\end{proof}
\begin{lem}
\label{lem:C0(a) is infinite}For every $a\in G$ the group $C^{0}\left(a\right)$
is infinite. Moreover, $\left[C\left(a\right):C^{0}\left(a\right)\right]<\infty$. 
\end{lem}
\begin{proof}
By Fact \ref{fact:BaldwinSaxl}, there is some $N$ such that for
every $k<\omega$ and every $a_{i}\in G$ for $i<k$, $\bigcap\set{C\left(a_{i}\right)}{i<k}=\bigcap\set{C\left(a_{i}\right)}{i\in I_{0}}$
where $I_{0}\subseteq k$ is of size $\leq N$. Find $a_{1},\ldots,a_{N}\in H_{a}$
with $\bigcap\set{C\left(a_{i}\right)}{i<N}\cap C\left(a\right)$
of maximal index in $C\left(a\right)$ (this index is bounded by
$n_{*}^{N}$). Let $D=\bigcap\set{C\left(a_{i}\right)}{i<N}\cap C\left(a\right)$.
Then, for every $b\in H_{a}$, $C\left(b\right)\cap D$ equals to
some sub-intersection $D'$ of size $N$, but then $\left[C\left(a\right):D'\right]=\left[C\left(a\right):D\right]$
so $D'=D$ and hence  $\bigcap\set{C\left(b\right)}{b\in H_{a}}=D$
and in particular it is infinite and of finite index in $C\left(a\right)$.
\end{proof}

\begin{proof}
[Proof of Theorem \ref{thm:Main-omega-categorical}.]Split into two
cases. 

\begin{casenv}
\item The set $Y=\set{a\in G}{a\in C^{0}\left(a\right)}$ is infinite. 

In this case, note that if $a,b\in Y$ then by Main Lemma \ref{lem:Main Lemma},
we may assume that $C^{0}\left(a\right)\subseteq C^{0}\left(b\right)\subseteq C\left(b\right)$
(because $b\in H_{b}$). But then $a\in C\left(b\right)$, so $Y$
is an infinite commutative $\emptyset$-definable set. Hence the group
generated by $Y$ must be abelian, and it must be $G$ by our choice
of $G$, so we are done. 
\item The set $Y$ is finite. 

Pick some $a_{0}\notin Y$. By induction on $n<\omega$, choose $a_{n}\in C^{0}\left(a_{n-1}\right)\backslash Y$.
We can find such elements by Lemma \ref{lem:C0(a) is infinite}. For
$n<\omega$, if $C^{0}\left(a_{n}\right)\subseteq C^{0}\left(a_{n+1}\right)$
then $a_{n+1}\in C^{0}\left(a_{n+1}\right)$ which cannot be, so by
Main Lemma \ref{lem:Main Lemma}, $C^{0}\left(a_{n+1}\right)\subseteq C^{0}\left(a_{n}\right)$. 

Let $K>\left[C\left(a_{n}\right):C^{0}\left(a_{n}\right)\right]$
for all $n<\omega$. As $a_{K}\in C^{0}\left(a_{i}\right)$ for all
$i<K$, $a_{i}\in C\left(a_{K}\right)$. Hence for some $i<j<K$,
$a_{i}^{-1}a_{j}\in C^{0}\left(a_{K}\right)\subseteq C^{0}\left(a_{i}\right)$.
But $a_{j}\in C^{0}\left(a_{i}\right)$ as well, so $a_{i}\in C^{0}\left(a_{i}\right)$
\textemdash{} contradiction. 
\end{casenv}
\end{proof}

\subsection{Concluding remarks}
\begin{problem}
Can we generalize this result to work under weaker assumptions than
$\omega$-categoricity, such as elimination of $\exists^{\infty}$?
(i.e., assume that $G$ is a dp-minimal group eliminating $\exists^{\infty}$,
also for imaginaries, with an infinite abelian subgroup, then does
it contain an infinite $\emptyset$-definable subgroup?)
\end{problem}

\begin{problem}
\label{prob:abelian by finite}Can one improve this to showing that
every dp-minimal $\omega$-categorical group is abelian-by-finite?
\end{problem}
\begin{rem}
Any abelian-by-finite group is stable, so if the we can solve Problem
\ref{prob:abelian by finite} positively, then it would mean that
any dp-minimal omega-categorical group is stable. Why? suppose that
$\left(G,\cdot\right)$ is a group with $H\leq G$ abelian of finite
index. Then there is a normal subgroup of $H$ with finite index in
$G$ (this is a standard exercise in group theory), and so we may
assume that $H$ is normal. Let $R=\Zz\left[G\right]$, the group
ring of $G$ over $\Zz$, whose elements we write as sums $\sum_{i<n}a_{i}g_{i}$
where $a_{i}\in\Zz$ and $g_{i}\in G$. Put a structure of a $\Zz\left[G\right]$-module
on $H$ by letting $\left(\sum_{i<n}a_{i}g_{i}\right)\cdot h=\sum_{i<n}a_{i}\cdot h^{g_{i}}$
(where $h^{g}=g^{-1}hg$). As a module, $H$ is stable (see \cite[Example 8.6.6]{TentZiegler}).
Now, $G$ can be interpreted in this structure (with parameters).
How? Suppose $\left[G:H\right]=n$. Then $G$ is the union of $g_{i}H$
where $\set{g_{i}}{i<n}$ are representatives for the different cosets
of $H$ in $G$. For each $i,j<n$ there is a unique $k\left(i,j\right)<n$
and $h\left(i,j\right)\in H$ such that $g_{i}\cdot g_{j}=g_{k\left(i,j\right)}h\left(i,j\right)$.
So now interpret $G$ as $\set{c_{i}}{i<n}\times H$, where the $c_{i}$'s
are distinct elements from $H$ and the product is given by $\left(c_{i},h\right)\cdot\left(c_{j},h'\right)=\left(c_{k\left(i,j\right)},h\left(i,j\right)h^{g_{j}}h'\right)$.
Since $h^{g_{j}}$ is just $g_{j}\cdot h$ in the module, this group
is definable in $H$. The map $\left(c_{i},h\right)\mapsto g_{i}h$
is then an isomorphism from this group to $G$. 
\end{rem}
\begin{problem}
Is there a (pure) group $\left(G,\cdot\right)$ which is $\omega$-categorical
NIP and unstable? 
\end{problem}

\begin{problem}
\label{prob:What-about-inp-minimal}What about inp-minimal groups?
Inp-minimality is the analogous notion to dp-minimality for $\mbox{NTP}_{2}$
\cite{MR3129735}, so it makes sense that this result still holds
there, as it does in both the simple (by \cite{MR975912}) and NIP
case. 
\end{problem}
\begin{rem}
Problem \ref{prob:What-about-inp-minimal} was solved by Frank Wagner,
who shared his proof with us in a private communication. We decided
to still include the following discussion, as it might be useful for
any future generalizations to $\mbox{NTP}_{2}$. 
\end{rem}
Using the same notation as in \cite{MR2898712}, we let $B\left(F\right)$
be the group of all continuous functions from the cantor space $2^{\omega}$
into a finite simple non-abelian group $F$. We also let $B^{-}\left(F\right)$
be the group of all such functions sending a fixed point $x_{0}\in2^{\omega}$
to $e\in F$. By \cite[Theorem 2.3]{MR679175,MR700288} they are characteristically
simple and $\omega$-categorical, and in fact by a theorem of Wilson
\cite{MR679175}, a countably infinite $\omega$-categorical characteristically
simple group is either isomorphic to one of them, is an abelian $p$-group
or is a perfect $p$-group. Neither groups is nilpotent-by-finite.
If they were nilpotent-by-finite, then there would be a normal nilpotent
subgroup of finite index, so they would be nilpotent (by Corollary
\ref{cor:Fitt}). But then they must be abelian, which they are not.

It is worthwhile to note the following.
\begin{prop}
\label{prop:B(F) is TP2}For a finite simple non-abelian group $F$,
both $B\left(F\right)$ and $B^{-}\left(F\right)$ have $\TPT$ and
in particular are not inp-minimal. 
\end{prop}
For the proof we will need the following simple criterion for having
$\TPT$.
\begin{lem}
\label{lem:Criterion for TP2}Suppose that $A$ is some infinite set
in $\C$ and $\varphi\left(x,y\right)$ is a formula such that for
some $k<\omega$ , for every sequence $\sequence{A_{i}}{i<\omega}$
of pairwise disjoint subsets of $A$, there are $\sequence{b_{i}}{i<\omega}$
such that $A_{i}\subseteq\varphi\left(\C,b_{i}\right)$ and $\set{\varphi\left(x,b_{i}\right)}{i<\omega}$
is $k$-inconsistent. Then $T$ has $\TPT$. 
\end{lem}
\begin{proof}
We may enumerate $A$ as $\sequence{a_{s}}{s\in\omega^{\omega}\wedge\left|\supp\left(s\right)\right|<\omega}$,
where $\supp\left(s\right)=\set{i\in\omega}{s\left(i\right)\neq0}$.
Let $A_{i,j}=\set{a_{s}}{s\left(i\right)=j}$. Then for each $i<\omega$,
$\set{A_{i,j}}{j<\omega}$ are mutually disjoint. By assumption we
can find $b_{i,j}$ for $i,j<\omega$ such that $\set{\varphi\left(x,b_{i,j}\right)}{j<\omega}$
are $k$-inconsistent and $A_{i,j}\subseteq\varphi_{i,j}\left(\C,b_{i,j}\right)$.
Then $\sequence{\varphi\left(x,b_{i,j}\right)}{i,j<\omega}$ witness
the tree property of the second kind. 
\end{proof}

\begin{proof}
We do the proof for $B\left(F\right)$. The proof for $B^{-}\left(F\right)$
is similar. 

Let $\varphi\left(x,y\right)$ be the formula $x\neq e$ and $x\in C\left(C\left(y\right)\right)$. 

Fix some $g\in F$, $g\neq e_{F}$. Suppose that $s\subseteq2^{\omega}$
is a clopen subset, and let $f_{s}\in B\left(F\right)$ be such that
$f_{s}\upharpoonright s$ is constantly $g$ and $f\upharpoonright2^{\omega}\backslash s$
is constantly $e_{F}$. Then $C\left(f_{s}\right)$ contains (in fact
equals) all functions $f'$ such that $f'\left(s\right)\subseteq C_{F}\left(g\right)$
(so outside of $s$ there are no restrictions on $f'$). Hence if
$f'\in C\left(C\left(f_{s}\right)\right)$ then $f'\upharpoonright2^{\omega}\backslash s$
is constantly $e_{F}$. 

It follows that if $s_{1}\cap s_{2}=\emptyset$ are two clopen subsets,
then $C\left(C\left(f_{s_{1}}\right)\right)\cap C\left(C\left(f_{s_{2}}\right)\right)=\left\{ e\right\} $. 

On the other hand, if $s_{1}\subseteq s_{2}$, then $f_{s_{1}}\in C\left(C\left(f_{s_{2}}\right)\right)$. 

Fix a sequence of pairwise disjoint clopen sets $\sequence{s_{i}}{i<\omega}$.
Then we see that for any choice of finite pairwise disjoint subsets
$A_{n}$, $n<\omega$ such that $A_{n}$ are finite, $f_{n}=f_{\bigcup\set{s_{i}}{i\in A_{n}}}$
satisfies $\set{\varphi\left(x,f_{n}\right)}{n<\omega}$ is $2$-inconsistent
but $f_{s_{i}}\models\varphi\left(x,f_{n}\right)$ if $i\in A_{n}$.
By compactness, we get such $f_{n}$'s for every choice of pairwise
disjoint subsets $A_{n}$ for $n<\omega$ (not necessarily finite).
By Lemma \ref{lem:Criterion for TP2} we are done. 
\end{proof}

\subsection{\label{subsec:C(A) abelian for finite A}A theorem on NIP groups}

We end this section with a general remark on NIP groups (without any
other assumptions).
\begin{thm}
\label{thm:always infinite abelian group}Suppose that $\left(G,\cdot\right)$
is an NIP group (or more generally, a type-definable group in an NIP
theory). Then there is some finite set $A$ such that $C\left(A\right)$
is abelian.
\end{thm}
\begin{proof}
Assume that $G$ is a type-definable group, defined by the type $\pi\left(x\right)$
(the multiplication $\cdot_{G}$ and the unit $e_{G}$ are definable).

Let $M$ be any $\left|\pi\right|^{+}$-saturated model, so that $G\left(M\right)\prec G\left(\C\right)$
by the Tarski-Vaught test.

Let $p_{0}$ be a partial type containing the formulas $x\in C\left(A\right)$
for all finite sets $A$ of $G\left(M\right)$. The partial type $p_{0}$
is finitely satisfiable in $G\left(M\right)$ (witnessed by $e_{G}$).
Let $S$ be the set of all global types in $S_{G}\left(\C\right)$
containing $p_{0}$ and f.s. in $G\left(M\right)$. All of these types
are in particular invariant over $M$, so their product is well defined.
(For the precise definition of a product of global invariant types,
see \cite[2.2.1]{pierrebook}, but one can understand it from the
proof.)

\begin{claim*}
For $p,q\in S$, $p\left(x\right)\otimes q\left(y\right)\models x\cdot y=y\cdot x$. 
\end{claim*}
\begin{proof}
We need to show that if $N\supseteq M$, $a\models q|_{N}$, $b\models p|_{Na}$
then $a\cdot b=b\cdot a$. If not, then $b\notin C\left(a\right)$,
so for some $b_{0}\in G\left(M\right)$, $b_{0}\notin C\left(a\right)$,
so $a\notin C\left(b_{0}\right)$ \textemdash{} contradiction.
\end{proof}
By \cite[Lemma 2.26]{pierrebook}, it follows that for any $a\models p,b\models q$,
$a\cdot b=b\cdot a$ (the proof there works just fine for type-definable
groups, because it only uses that the formula for multiplication is
NIP, but multiplication is definable). 

By compactness for every $p\left(x\right),q\left(y\right)\in S$ there
are formulas $\psi_{p,q}\left(x\right)\in p,\varphi_{p,q}\left(y\right)\in q$
such that for every $a\models\psi_{p,q}$, $b\models\varphi_{p,q}$
in $G$, $a\cdot b=b\cdot a$. Fix $p$. By compactness (as $S$ is
closed), there is a finite set of types $q_{i}\in S$ for $i<n$ such
that $\varphi_{p}=\bigvee_{i<n}\varphi_{p,q_{i}}$ contains $S$.
Let $\psi_{p}=\bigwedge_{i<n}\psi_{p,q_{i}}$. Again by compactness
there are $p_{i}$ for $i<m$ such that $\bigvee_{i<m}\psi_{p_{i}}$
contains $S$. Let $\chi=\left(\bigwedge_{i<m}\varphi_{p_{i}}\right)\wedge\left(\bigvee_{i<m}\psi_{p_{i}}\right)$,
then $\chi$ contains $S$ and for every $a,b\models\chi$ in $G\left(\C\right)$,
$a\cdot b=b\cdot a$. (This is the same as in the proof of \cite[Proposition 2.27]{pierrebook}.)

It cannot be that for all finite $A\subseteq G\left(M\right)$, $\neg\chi\left(M\right)\cap C\left(A\right)\neq\emptyset$
(otherwise we can define a type, f.s. in $M$, containing $p_{0}$,
so in $S$, but not satisfying $\chi$). Hence there is some finite
$A\subseteq G\left(M\right)$ such that $C\left(A\right)\left(M\right)\models\chi$.
Hence $C\left(A\right)\left(M\right)$ is abelian, but as $G\left(M\right)\prec G\left(\C\right)$,
so is $C\left(A\right)$.  
\end{proof}
\begin{rem}
When the group $G$ is an $\omega$-categorical characteristically
simple NIP group, then by Proposition \ref{prop:B(F) is TP2}, and
the remark before it (or just \cite[Fact 0.1 and Proposition 3.2]{MR2898712}),
Krupinski's proof of \cite[Proposition 3.1]{MR2898712} gives us that
for any finite set $A$, $C\left(A\right)$ is infinite. Together
with Theorem \ref{thm:always infinite abelian group}, we know that
we can find some $A$ such that $C\left(A\right)$ is abelian and
infinite. 
\end{rem}

\section{\label{sec:Boundedly-many-global}$f$-generic is the same as strongly
$f$-generic}

Assume that $G$ is definably-amenable and dp-minimal. The main theorem
here says that any definable $X\subseteq G$ which divides over a
small model also $G$-divides. This means that that are at most boundedly
many global $f$-generic types and a global type is $f$-generic iff
it is strongly $f$-generic (see Corollaries \ref{cor:f-generic implies strongly f-generic}
and \ref{cor:boundedly many types}). 

Let us first recall the definitions. Throughout we assume $T$ is
NIP, and we work in a monster model $\C$.
\begin{defn}
A definable group $G$ is definably amenable if it admits a $G$-invariant
Keisler measure on its definable subsets. 
\end{defn}
A \emph{Keisler measure} is a finitely additive probability measure
on definable subsets of $G$. We will not use this definition, so
there is no need for us to get too deeply into Keisler measures. Instead
we will use the following characterization from \cite{ChernikovSimonDA2015}
given in terms of $G$-dividing. 
\begin{defn}
For $X\subseteq G$ definable, we say that it $G$-divides if there
is an indiscernible sequence $\sequence{g_{i}}{i<\omega}$ of elements
from $G$ over the parameters defining $X$ such that $\set{g_{i}X}{i<\omega}$
is inconsistent (equivalently, remove the indiscernibility assumption
and replace it with $k$-inconsistency). Similarly, we say that $X$
\emph{right-$G$-divides} if there is a sequence as above such that
$\sequence{Xg_{i}}{i<\omega}$ is inconsistent. 
\end{defn}
\begin{fact}
\label{fact:ideal of G-dividing}\cite[Corollary 3.5]{ChernikovSimonDA2015}Let
$G$ be a group definable in an NIP theory. Then if $G$ is definably
amenable then the family of $G$-dividing subsets of $G$ forms an
ideal. Hence in this case any non-$G$-dividing partial type can be
extended to a global one. 
\end{fact}
As an example which relates to the previous section, we note that
any countable $\omega$-categorical group is locally finite and hence
it is amenable by \cite[Example 1.2.13]{MR1874893} and so any group
elementarily equivalent to it is definably amenable \cite[Example 8.13]{pierrebook}.
\begin{defn}
A global type is called\emph{ $f$-generic} if it contains no $G$-dividing
formula. 
\end{defn}
\begin{rem}
\cite[Proposition 3.4]{ChernikovSimonDA2015}If $G$ is definably
amenable, then $p$ is $f$-generic iff all its formula are $f$-generic,
which means that for every $\varphi\in p$, no translate of $\varphi$
forks over $M$ where $M$ is some small model containing the parameters
of $\varphi$. It is also proved there that a formula is $f$-generic
iff it does not $G$-divide, so we will use these terms interchangeably.
Similarly, we will write right-$f$-generic for non-right-$G$-dividing. 
\end{rem}
\begin{fact}
\label{fact:f-generic =00003D G00 invariant}\cite[Proposition 3.9]{ChernikovSimonDA2015}When
$G$ is definably amenable then a global type is $f$-generic type
iff it is $G^{00}$-invariant. 
\end{fact}
\begin{thm}
\label{thm:bddly many f-generics}Suppose that $G$ is dp-minimal
and definably amenable.  Then if $\varphi\left(x,c\right)$ forks
over a small model $M$, then it $G$-divides. 
\end{thm}
\begin{proof}
Suppose not.

By assumption (and as forking equals dividing over models, see \cite{Kachernikov}),
there is an $M$-indiscernible sequence $\sequence{c_{j}}{j<\omega}$
such that $\sequence{\varphi\left(x,c_{j}\right)}{j<\omega}$ is inconsistent. 

However, $\varphi\left(x,c_{j}\right)$ is still $f$-generic. 

We now divide into two cases: either there is a formula $\psi_{0}\left(x,b\right)$
which is right-$f$-generic but not $f$-generic (call this case 0),
or not (case 1). In case 0, let $\zeta_{0}\left(x,y,b\right)=\psi_{0}\left(y^{-1}x,b\right)$.

Note that if case 0 does not occur, then every $G$-dividing formula
also right-$G$-divides. As $X$ is $f$-generic iff $X^{-1}$ is
right-$f$-generic, this means that every right-$G$-dividing formula
also $G$-divides. 

In case 1, choose a formula $\psi_{1}\left(x,b\right)$ for which,
for every formula $\chi\left(x\right)\supseteq G^{00}$ (with no parameters)
both $\psi_{1}\left(x,b\right)\wedge\chi\left(x\right)$ and $\neg\psi_{1}\left(x,b\right)\wedge\chi\left(x\right)$
are $f$-generic (such a formula exists, as otherwise there is a unique
non-$G$-dividing type concentrating on $G^{00}$, so the number of
$f$-generic types is bounded, but by assumption there are unboundedly
many). Let $\zeta_{1}\left(x,y,z,b\right)=\psi_{1}\left(y^{-1}x,b\right)\wedge\neg\psi_{1}\left(z^{-1}x,b\right).$
We may assume that $\sequence{c_{j}}{i<\omega}$ is indiscernible
over $Mb$. 

Depending on the case, let $\zeta\left(x,y,z,b\right)$ be either
$\zeta_{0}$ or $\zeta_{1}$ (so $z$ might be redundant). Construct
a sequence $\sequence{I_{i},g_{i},h_{i}}{i<\omega}$ such that:

\begin{itemize}
\item In case 0, $g_{i}\in G$. In case 1, $g_{i},h_{i}\in G^{00}$. 
\item $I_{i}$ is indiscernible, $I_{i}=\sequence{e_{i,j}}{j<\omega}$,
$e_{i,j}\models\varphi\left(x,c_{j}\right)$ for all $i,j<\omega$.
\item $e_{i,j}\models\zeta\left(x,g_{i},h_{i},b\right)$ for all $i,j<\omega$.
\item $e_{i',j}\not\models\zeta\left(x,g_{i},h_{i},b\right)$ for all $i',i,j<\omega$
whenever $i'>i$. 
\end{itemize}
(In case 0, we only need $g_{i}$.) How? 

Note that for any $g,h\in G^{00}$, $\zeta\left(x,g,h,b\right)$ does
$G$-divide by Fact \ref{fact:f-generic =00003D G00 invariant} (this
is trivially true in case 0). 

By compactness it is enough to construct such a sequence for $i<n$.
Suppose we have $\sequence{I_{i},g_{i},h_{i}}{i<n}$. Let $\xi\left(x\right)=\bigvee_{i<n}\zeta\left(x,g_{i},h_{i},b\right)$.
Then $\xi\left(x\right)$ does $G$-divide by Fact \ref{fact:ideal of G-dividing}.
Hence $\varphi\left(x,c_{j}\right)\backslash\xi\left(x\right)$ is
not empty for all $j$, and hence we may find a sequence $e_{n,j}\models\varphi\left(x,c_{j}\right)\backslash\xi\left(x\right)$.
Consider the sequence $I=\sequence{\left(e_{0,j},\ldots,e_{n,j},c_{j}\right)}{j<\omega}$.
By Ramsey and compactness there is an $Mh_{<n}g_{<n}b$-indiscernible
sequence $I'$ with the same EM-type as $I$ over over $Mh_{<n}g_{<n}b$.
There is an automorphism taking $\sequence{c_{j}'}{j<\omega}$ to
$\sequence{c_{j}}{j<\omega}$ over $Mb$, and applying it we are in
the same situation as before (changing $h_{<n}g_{<n}$ and $e_{i,j}$)
but now $I_{n}=\sequence{e_{n,j}}{j<\omega}$ is indiscernible. This
takes cares of all the bullets except the third one, for which needs
to find $g_{n},h_{n}$. 

In case 0, the set $\set{\psi_{0}\left(x,b\right)\cdot e_{n,j}^{-1}}{j<\omega}$
is consistent (as $\psi_{0}\left(x,b\right)$ does not right-$G$-divide),
so contains some $g\in G$, hence $ge_{n,j}\models\psi_{0}\left(x,b\right)$,
i.e., $e_{n,j}\models g^{-1}\cdot\psi_{0}\left(x,b\right)=\zeta_{0}\left(x,g^{-1},b\right)$
so let $g_{n}=g^{-1}$. 

In case 1, the set $\set{\left(\chi\left(x\right)\wedge\psi_{1}\left(x,b\right)\right)\cdot e_{n,j}^{-1}}{G^{00}\subseteq\chi\left(x\right),j<\omega}$
is consistent (as $G$-dividing = right-$G$-dividing in this case)
so again we can find $g\in G^{00}$ realizing it, so in particular
$e_{n,j}\models g^{-1}\cdot\psi_{1}\left(x,b\right)$. Similarly,
there is some $h\in G^{00}$ such that $e_{n,j}\models h^{-1}\cdot\left(\neg\psi_{1}\left(x,b\right)\right)$.
Finally, choose $g_{n}=g^{-1}$ and $h_{n}=h^{-1}$. 

This finishes the construction. 

Now by Ramsey and compactness we may assume that $\sequence{I_{i}g_{i}h_{i}}{i<\omega}$
is indiscernible over $Mb\sequence{c_{j}}{j<\omega}$ and that $\sequence{\sequence{e_{i,j}}{i<\omega}c_{j}}{j<\omega}$
is indiscernible over $Mb\sequence{g_{i}h_{i}}{i<\omega}$. 

Let $\xi\left(x,y,y',z,z',b\right)=\zeta\left(x,y,z,b\right)\backslash\zeta\left(x,y',z',b\right)$.
Then for every $i>0$ and $j<\omega$, $e_{i_{0},j_{0}}\models\xi\left(x,g_{i},h_{i},g_{i-1},h_{i-1}\right)$
iff $i_{0}=i$ and $e_{i_{0},j_{0}}\models\varphi\left(x,c_{j}\right)$
iff $j=j_{0}$ contradicting dp-minimality. 
\end{proof}
For the next corollary, we recall that in the context of NIP, definably
amenable groups, a global type is called\emph{ strongly $f$-generic}
if it is $f$-generic and does not fork over some small model (this
is not the original definition, but see \cite[Proposition 3.10]{ChernikovSimonDA2015}. 
\begin{cor}
\label{cor:f-generic implies strongly f-generic}If $G$ is a dp-minimal
definably amenable group, then any global $f$-generic $1$-type $p$
is strongly $f$-generic. 
\end{cor}
\begin{proof}
Take any small model $M$. Then $p$ cannot divide over $M$. 
\end{proof}
\begin{rem}
Theorem \ref{thm:bddly many f-generics} does not hold for a group
definable in a dp-minimal theory. Consider $T=RCF$, and let $G=R^{2}$
where $R$ is a saturated model of $T$. Example 3.11 in \cite{ChernikovSimonDA2015}
gives a $G$-invariant type $r\left(x,y\right)$ (so does not $G$-divide)
which is not invariant over any small model $M$.
\end{rem}
\begin{cor}
\label{cor:boundedly many types}Suppose that $G$ is a dp-minimal
definably amenable group. Then there are boundedly many global $f$-generic
types.
\end{cor}
\begin{proof}
Fix some small model $M$. This follows by NIP and (the proof of)
Corollary \ref{cor:f-generic implies strongly f-generic}, as there
are boundedly many global types non-forking over $M$ (by NIP they
must be invariant over $M$). 
\end{proof}

\begin{cor}
Suppose $G$ is a dp-minimal definably amenable group. Then there
are boundedly many $G$-invariant Keisler measures. 
\end{cor}
\begin{proof}
Suppose that there are unboundedly many such measures $\mu_{i}$.
Fix some small model $M$. By Erd\"os-Rado, we may find a formula
$\varphi\left(x,y\right)$, a type $p\in S_{y}\left(M\right)$ some
numbers $\alpha\neq\beta\in\left[0,1\right]$ and a sequence of $G$-invariant
Keisler measures $\sequence{\mu_{i}}{i<\omega}$ such that for all
$i<j<\omega$, $\mu_{i}\left(\varphi\left(x,a_{i,j}\right)\right)=\alpha$
and $\mu_{j}\left(\varphi\left(x,a_{i,j}\right)\right)=\beta$ for
some $a_{i,j}\models p$ in $\C$. 

Then there are $a,b\models p$ such that $\alpha=\mu_{0}\left(\varphi\left(x,a\right)\right)\neq\mu_{1}\left(\varphi\left(x,a\right)\right)=\beta$
and $\alpha=\mu_{1}\left(\varphi\left(x,b\right)\right)\neq\mu_{2}\left(\varphi\left(x,b\right)\right)=\beta$.
In particular $\mu_{1}\left(\varphi\left(x,a\right)\mathrela{\triangle}\varphi\left(x,b\right)\right)\neq0$.
As $\mu_{1}$ is $G$-invariant, $\varphi\left(x,a\right)\mathrela{\triangle}\varphi\left(x,b\right)$
does not $G$-divide (see \cite[Theorem 3.38]{ChernikovSimonDA2015},
but this follows easily from the definitions), but it forks by NIP. 
\end{proof}
\begin{cor}
Suppose that $\left(F,+,\cdot\right)$ is a dp-minimal field. Then
every additive $f$-generic set (i.e., with respect to $\left(F,+,0\right)$)
is also multiplicatively $f$-generic. 
\end{cor}
\begin{proof}
Note that any abelian group is definably amenable (see \cite[Example 8.13]{pierrebook}).

Suppose that $X$ is additively $f$-generic. For any $a\in F^{\times}$,
$a\cdot X$ is also additively $f$-generic. Hence, if $X$ is not
multiplicatively $f$-generic, then there is an indiscernible sequence
$\sequence{a_{i}}{i<\omega}$ over the parameters defining $X$ such
that $\set{a_{i}X}{i<\omega}$ is inconsistent (so $k$-inconsistent
for some $k<\omega$). Increasing the sequence to any length, by Fact
\ref{fact:ideal of G-dividing}, we get unboundedly many additively
$f$-generic global types. Contradicting Corollary \ref{cor:boundedly many types}.
\end{proof}
\begin{problem}
Is there a dp-minimal group which is not definably amenable?
\end{problem}
\bibliographystyle{alpha}
\bibliography{common2}

\begin{thebibliography}{BCM79}

\bibitem[AM97]{MR1426520}
Richard Archer and Dugald Macpherson.
\newblock Soluble omega-categorical groups.
\newblock {\em Math. Proc. Cambridge Philos. Soc.}, 121(2):219--227, 1997.

\bibitem[App83]{MR700288}
A.~B. Apps.
\newblock On the structure of {$\aleph _{0}$}-categorical groups.
\newblock {\em J. Algebra}, 81(2):320--339, 1983.

\bibitem[BCM79]{MR533805}
Walter Baur, Gregory Cherlin, and Angus Macintyre.
\newblock Totally categorical groups and rings.
\newblock {\em J. Algebra}, 57(2):407--440, 1979.

\bibitem[BS76]{BaSxl}
J.~T. Baldwin and Jan Saxl.
\newblock Logical stability in group theory.
\newblock {\em J. Austral. Math. Soc. Ser. A}, 21(3):267--276, 1976.

\bibitem[Che14]{MR3129735}
Artem Chernikov.
\newblock Theories without the tree property of the second kind.
\newblock {\em Ann. Pure Appl. Logic}, 165(2):695--723, 2014.

\bibitem[CK12]{Kachernikov}
Artem Chernikov and Itay Kaplan.
\newblock Forking and dividing in $\operatorname{NTP}_{\operatorname{2}}$
  theories.
\newblock {\em Journal of Symbolic Logic}, 77(1):1--20, 2012.

\bibitem[CS15]{ChernikovSimonDA2015}
Artem Chernikov and Pierre Simon.
\newblock Definably amenable nip groups.
\newblock 2015.
\newblock \url{arXiv:1502.04365}.

\bibitem[DGL11]{MR2822489}
Alfred Dolich, John Goodrick, and David Lippel.
\newblock Dp-minimality: basic facts and examples.
\newblock {\em Notre Dame J. Form. Log.}, 52(3):267--288, 2011.

\bibitem[DK13]{MR3037553}
Jan Dobrowolski and Krzysztof Krupi{\'n}ski.
\newblock On {$\omega$}-categorical, generically stable groups and rings.
\newblock {\em Ann. Pure Appl. Logic}, 164(7-8):802--812, 2013.

\bibitem[Fel78]{MR0491151}
Ulrich Felgner.
\newblock {$\aleph _{0}$}-categorical stable groups.
\newblock {\em Math. Z.}, 160(1):27--49, 1978.

\bibitem[Hod93]{Hod}
Wilfrid Hodges.
\newblock {\em Model Theory}, volume~42 of {\em Encyclopedia of mathematics and
  its applications}.
\newblock Cambridge University Press, Great Britain, 1993.

\bibitem[Joh15]{Johnson2015}
Will Johnson.
\newblock On dp-minimal fields.
\newblock 2015.
\newblock \url{arXiv:1507.02745}.

\bibitem[JSW15]{JahnkeSimonWalsberg2015}
Franziska Jahnke, Pierre Simon, and Erik Walsberg.
\newblock Dp-minimal valued fields.
\newblock 2015.
\newblock \url{arXiv:1507.03911}.

\bibitem[KOU13]{MR3091265}
Itay Kaplan, Alf Onshuus, and Alexander Usvyatsov.
\newblock Additivity of the dp-rank.
\newblock {\em Trans. Amer. Math. Soc.}, 365(11):5783--5804, 2013.

\bibitem[Kru12]{MR2898712}
Krzysztof Krupi{\'n}ski.
\newblock On {$\omega$}-categorical groups and rings with {NIP}.
\newblock {\em Proc. Amer. Math. Soc.}, 140(7):2501--2512, 2012.

\bibitem[KS14]{KaplanSimon}
Itay Kaplan and Pierre Simon.
\newblock Witnessing dp-rank.
\newblock {\em Notre Dame J. Form. Log.}, 55(3):419--429, 2014.

\bibitem[KW73]{MR0470081}
Otto~H. Kegel and Bertram A.~F. Wehrfritz.
\newblock {\em Locally finite groups}.
\newblock North-Holland Publishing Co., Amsterdam-London; American Elsevier
  Publishing Co., Inc., New York, 1973.
\newblock North-Holland Mathematical Library, Vol. 3.

\bibitem[Mac71]{Macintyre}
Angus Macintyre.
\newblock On {$\omega _{1}$}-categorical theories of fields.
\newblock {\em Fund. Math.}, 71(1):1--25. (errata insert), 1971.

\bibitem[Mac88]{MR975912}
H.~D. Macpherson.
\newblock Absolutely ubiquitous structures and {$\aleph_0$}-categorical groups.
\newblock {\em Quart. J. Math. Oxford Ser. (2)}, 39(156):483--500, 1988.

\bibitem[OU11]{OnUs1}
Alf Onshuus and Alexander Usvyatsov.
\newblock On dp-minimality, strong dependence and weight.
\newblock {\em Journal of symbolic logic}, 76(3):737--758, 2011.

\bibitem[Rob93]{MR1261639}
Derek J.~S. Robinson.
\newblock {\em A course in the theory of groups}, volume~80 of {\em Graduate
  Texts in Mathematics}.
\newblock Springer-Verlag, New York, 1993.

\bibitem[Run02]{MR1874893}
Volker Runde.
\newblock {\em Lectures on amenability}, volume 1774 of {\em Lecture Notes in
  Mathematics}.
\newblock Springer-Verlag, Berlin, 2002.

\bibitem[She08]{Sh876}
Saharon Shelah.
\newblock Minimal bounded index subgroup for dependent theories.
\newblock {\em Proc. Amer. Math. Soc.}, 136(3):1087--1091 (electronic), 2008.

\bibitem[She09]{Sh783}
Saharon Shelah.
\newblock Dependent first order theories, continued.
\newblock {\em Israel J. Math.}, 173:1--60, 2009.

\bibitem[Sim11]{Simon-Dp-min}
Pierre Simon.
\newblock On dp-minimal ordered structures.
\newblock {\em Journal of Symbolic Logic}, 76(2):448--460, 2011.

\bibitem[Sim14]{MR3343528}
Pierre Simon.
\newblock Dp-minimality: invariant types and dp-rank.
\newblock {\em J. Symb. Log.}, 79(4):1025--1045, 2014.

\bibitem[Sim15]{pierrebook}
Pierre Simon.
\newblock {\em A Guide to NIP Theories (Lecture Notes in Logic)}.
\newblock Cambridge University Press, 2015.

\bibitem[TZ12]{TentZiegler}
Katrin Tent and Martin Ziegler.
\newblock {\em A course in model theory}, volume~40 of {\em Lecture Notes in
  Logic}.
\newblock Association for Symbolic Logic, La Jolla, CA; Cambridge University
  Press, Cambridge, 2012.

\bibitem[Wil82]{MR679175}
J.~S. Wilson.
\newblock The algebraic structure of {$\aleph _{0}$}-categorical groups.
\newblock In {\em Groups---{S}t. {A}ndrews 1981 ({S}t. {A}ndrews, 1981)},
  volume~71 of {\em London Math. Soc. Lecture Note Ser.}, pages 345--358.
  Cambridge Univ. Press, Cambridge-New York, 1982.

\end{thebibliography}

\end{document}